\documentclass{article}

\usepackage{amsmath}
\usepackage{amsthm}
\usepackage{natbib}
\usepackage{graphicx}
\usepackage[tight]{subfigure}
\usepackage{amssymb}
\usepackage{graphicx}
\usepackage[tight]{subfigure}

\newtheorem{proposition}{Proposition}
\newtheorem{theorem}{Theorem}

\newtheorem{lemma}{Lemma}

\newcommand{\vanden}{N}
\newcommand{\hatvap}{\hat p}
\newcommand{\hatvapUMVU}{\hat p_\mathrm{umvu}}
\newcommand{\hatvapML}{\hat p_\mathrm{ml}}

\renewcommand{\Pr}{P}
\DeclareMathOperator {\E}{E}
\newcommand{\nnum}{r}
\newcommand{\nden}{n}
\newcommand{\ndcero}{m}
\newcommand{\ndcerocuno}{\mu_1}
\newcommand{\ndcerocdos}{\mu_2}

\newcommand{\emphmio}{\emph}
\newcommand{\ff}[2]{#1^{(#2)}} 
\newcommand{\diff}{\mathrm d}
\newcommand{\risk}{\eta}
\newcommand{\fest}{g}
\newcommand{\pmax}{p_\mathrm{max}}
\newcommand{\pinf}{p_\mathrm l}
\newcommand{\psup}{p_\mathrm u}
\newcommand{\Omegail}{\tilde\Omega}
\newcommand{\deltail}{\tilde\delta}
\newcommand{\riskej}{\lambda}
\newcommand{\riskparte}{\theta}
\newcommand{\riskparteil}{\tilde\theta}

\begin{document}

\title{Estimation of a probability in inverse binomial sampling under normalized linear-linear\\ and inverse-linear loss}
\author{Luis Mendo \thanks{E.T.S. Ingenieros de Telecomunicaci\'on, Polytechnic University of Madrid, 28040 Madrid, Spain. E-mail: lmendo@grc.ssr.upm.es.}}
\date{\today}

\maketitle

\begin{abstract}
Sequential estimation of the success probability $p$ in inverse binomial sampling is considered in this paper. For any estimator $\hatvap$, its quality is measured by the risk associated with normalized loss functions of linear-linear or inverse-linear form. These functions are possibly asymmetric, with arbitrary slope parameters $a$ and $b$ for $\hatvap < p$ and $\hatvap > p$ respectively. Interest in these functions is motivated by their significance and potential uses, which are briefly discussed. Estimators are given for which the risk has an asymptotic value as $p \rightarrow 0$, and which guarantee that, for any $p \in (0,1)$, the risk is lower than its asymptotic value. This allows selecting the required number of successes, $\nnum$, to meet a prescribed quality irrespective of the unknown $p$. In addition, the proposed estimators are shown to be approximately minimax when $a/b$ does not deviate too much from $1$, and asymptotically minimax as $\nnum \rightarrow \infty$ when $a=b$.

\emph{Keywords:} Sequential estimation, Point estimator, Inverse binomial sampling, Asymmetric loss function.
\end{abstract}

\section{Motivation and considered loss functions}
\label{parte: motiv}

The estimation of the success probability $p$ of a sequence of Bernoulli trials is a recurring problem, arising in many branches of science and engineering. The quality of a point estimator of $p$, denoted as $\hatvap$, can be measured in terms of its \emphmio{risk}, or average loss associated with a certain \emphmio{loss function} $L$. Since a given error is most meaningful when compared with the true value $p$, quality measures used in practice are most often normalized ones \citep{Mendo10_env}. This corresponds to $L$ being a function of $\hatvap/p$, rather than of $\hatvap$. Common loss functions include normalized squared error $(\hatvap/p-1)^2$ and normalized absolute error $|\hatvap/p-1|$. Interval estimation can also be analyzed in terms of a certain loss function \citep[p.~64]{Berger85} such that the resulting risk is the confidence level associated with an estimation interval.

Fixed-sample approaches to this problem suffer from the drawback that the required size depends on the unknown parameter $p$, and thus cannot be determined in advance. Therefore a \emphmio{sequential procedure} is required, consisting of a \emphmio{stopping rule}, which yields a random sample size, and an estimator based on the observed sample.

Sequential estimation in Bernoulli trials has been studied by many authors.
\citet{Girshick46} introduce and analyze a general class of sequential procedures for the unbiased estimation of $p$.
Using a similar approach, \citet{DeGroot59} gives criteria for the selection of appropriate sampling plans for unbiased estimation of functions of $p$, with estimator performance measured by variance. His work shows that the fixed-size and inverse procedures are the only efficient sampling plans, i.e.~the only ones for which the Cramer-Rao bound (or information inequality) holds with equal sign.
\citet{Hubert00} focus on the asymptotic behaviour as $p \rightarrow 0$. In this setting, they consider estimation of powers of $p$, with a loss function given as squared error loss multiplied by another power of $p$. \citet{Baran10} carry out a similar asymptotic analysis for a different loss function.
Using a Bayesian approach, \citet{Cabilio75} and \citet{Cabilio77} consider symmetrized relative squared error loss (given as $[(\hatvap-p)/(p(1-p))]^2$) plus a fixed cost per observation, and find sequential procedures which minimize the Bayes risk in the estimation of $p$.
\citet{Alvo77} studies sequential Bayes estimation from a more general point of view, where the observations are not necessarily Bernoulli variables, and considering squared error loss.

A particularly appealing stopping rule, first discussed by \citet{Haldane45}, is \emphmio{inverse binomial sampling} (also known as negative binomial sampling). Given $\nnum \in \mathbb N$, this rule consists of taking as many observations as necessary to obtain exactly $\nnum$ successes. The random number of observations, $\vanden$, is a sufficient statistic for $p$ \citep[p.~101]{Lehmann98}. The interest in this stopping rule is motivated by the useful properties of the obtained estimators. Namely, it has been shown that for an estimator $\hatvap = \fest(\vanden)$ such that $\lim_{\nden \rightarrow \infty} \nden \fest(\nden)$ exists and is positive, and for a general class of loss functions defined by certain regularity conditions, the risk has an asymptotic value as $p \rightarrow 0$ \citep{Mendo10_env}. Moreover, estimators have been found whose risk for $p$ arbitrary is guaranteed not to exceed its asymptotic value, for the specific cases of normalized mean squared error \citep{Mikulski76} \citep{Sathe77}, normalized mean absolute error \citep{Mendo09a} and confidence associated with a relative interval \citep{Mendo06} \citep{Mendo08a} \citep{Mendo10}. This allows selecting an appropriate value of $\nnum$ that meets a prescribed risk irrespective of the unknown $p$.

In all cases mentioned in the preceding paragraph, the loss incurred by a negative error equals that of the corresponding positive error. In practice, however, situation-specific factors may render underestimation more or less costly than overestimation \citep{Christoffersen97} \citep{Akdeniz04}. Consider for example $p=0.01$ and two possible values of $\hatvap$, namely $0.019$ and $0.001$. The absolute error (normalized or otherwise) is the same for both values of the estimator, as is the squared error. Nevertheless, with the first estimate $\hatvap$ is $1.9$ times $p$, whereas with the second $p$ is $10$ times $\hatvap$. In many applications it may be advisable to assign a higher loss to the second estimate. With absolute error, this could be accomplished by generalizing the loss function to one with a different slope on each side. Denoting $x = \hatvap /p$, this generalized loss is given by
\begin{equation}
\label{eq: loss abs error gen}
L(x) = \begin{cases}
a(1-x) & \text{if } x \leq 1, \\
b(x-1) & \text{if } x > 1,
\end{cases}
\end{equation}
with parameters $a, b \geq 0$, $(a,b) \neq (0,0)$. This function, known as (normalized) \emphmio{linear-linear} loss,
frequently arises in applications; see for example \citet{Granger69} and \citet{Christoffersen97}. Another proposed function (not considered in this paper) which gives different weights to positive and negative errors is the \emphmio{linear-exponential} loss, whose normalized version is $L(x) = b[\exp(a(x-1))-a(x-1)-1]$, with parameters $a \neq 0$, $b>0$ \citep{Akdeniz04}. The ratio $a/b$, in the linear-linear loss, or the parameter $a$, in the linear-exponential, control the relative importance given to underestimation and overestimation. Note that in both cases the loss due to underestimation is bounded, unlike that of overestimation, which may be arbitrarily large.

In certain situations it may be meaningful to define loss as proportional to $\hatvap/p$ or $p/\hatvap$, whichever is largest. Thus with the values in the previous example, the loss would be proportional to $1.9$ and $10$ respectively. In the following, the function $s(\hatvap,p) = \max\{\hatvap/p,p/\hatvap\}$ will be referred to as the \emphmio{symmetric ratio} of $\hatvap$ and $p$ (the name is motivated by the fact that $s(\hatvap,p) = s(p,\hatvap)$). The loss thus defined is inherently normalized, because it only depends on $\hatvap$ and $p$ through $x=\hatvap/p$. Subtracting $1$ in order to have a minimum loss equal to $0$, the loss function is expressed as $L(x) = \max\{x,1/x\} - 1$. This function is unbounded for underestimation as well as for overestimation errors. In fact, its graph is symmetric about $x=1$ if $\hatvap$, or $x$, is represented in logarithmic scale (this is obvious if $L(x)$ is written as $\exp|\log x|-1$). The risk corresponding to this loss is the mean symmetric ratio minus $1$, and represents a normalized measure of dissimilarity between $\hatvap$ and $p$, with smaller values corresponding to better estimators. A generalization is obtained, as before, by allowing different multiplicative parameters $a,b \geq 0$, $(a,b) \neq (0,0)$ on each side of the function:
\begin{equation}
\label{eq: loss upper ratio gen}
L(x) = \begin{cases}
a(1/x-1) & \text{if } x \leq 1, \\
b(x-1) & \text{if } x > 1.
\end{cases}
\end{equation}%
This will be referred to as \emphmio{inverse-linear} loss.

The loss function \eqref{eq: loss upper ratio gen}, in addition to providing a natural measure of estimation quality, namely generalized mean symmetric ratio, can be representative of incurred cost in specific applications. In spite of this, it has not been used previously in the context of estimation problems, to the author's knowledge. As an example of application, consider the production of a certain device which is subject to manufacturing defects, such as image sensors for digital cameras. Several factors in the production process (such as the presence of dust particles) may result in a sensor with specific pixels systematically showing incorrect information. Since it would be too expensive to discard all sensors that have some defect, the commonly adopted solution is as follows. Each produced sensor is tested, and if the number of defective pixels is not too large it is accepted. The location of such pixels is permanently recorded in the camera, so that they can be corrected as a part of the processing applied by the camera to generate the image.

In high-quality camera models, however, it may be desirable to use sensors with an extremely low number of defects. A possible procedure is to classify each produced sensor as ``premium'' or ``standard'', depending on whether the number of pixel defects is extremely low or merely acceptable. Premium sensors are reserved for advanced cameras, which incorporate high-quality lenses, whereas standard sensors are mounted in consumer-level cameras with average-quality lenses. For ease of explanation, these two types of lenses will also be referred to as premium and standard, respectively. The production of each type of lens is a more deterministic process than that of sensors, and thus the number of produced lenses of each type is easily controlled.

It will be assumed that the manufacturer is primarily interested in its premium line of cameras. A number $S$ of sensors is to be produced, and the amount of premium lenses that will be required needs to be planned in advance. To this end, an estimate $\hatvap$ is made of the proportion $p$ of sensors that will turn out to be of the premium type (this can be done using inverse binomial sampling); and $S \hatvap$ premium lenses are made available. The actual proportion of premium sensors, $p$, may be lower than $\hatvap$, in which case some of the premium lenses will be left unused; or it may be greater, and then some of the premium sensors will not be used. In either case, some resources are wasted. If the cost associated with each unused part is $a$ for a sensor and $b$ for a lens, the risk computed from the loss function \eqref{eq: loss upper ratio gen} is the average cost of wasted resources per assembled premium camera unit.

The rest of the paper analyzes inverse binomial sampling under the loss functions \eqref{eq: loss abs error gen} and \eqref{eq: loss upper ratio gen}. The first has already been analyzed for the particular case $a=b$ by \citet{Mendo09a}, and the generalization to $a \neq b$ will be seen to be rather straightforward. The second function has not been dealt with before, to the author's knowledge, and its analysis turns out to be more difficult. Although the main focus of the paper is on the second, results for the first are also interesting by themselves. In each case, estimators are given in Section \ref{parte: res} such that the risk for $p \in (0,1)$ is guaranteed to be lower than its asymptotic value. Section \ref{parte: disc} discusses these results and makes a comparison with the optimum performance that could be achieved by using other estimators. It is shown that the proposed estimators are approximately minimax if $a/b$ is close to $1$; and for $a=b$ they are asymptotically minimax as $\nnum \rightarrow \infty$. Section \ref{parte: proofs} contains the proofs to all results.

\section{Main results}
\label{parte: res}

Consider a sequence of Bernoulli trials with probability of success $p$, and a random stopping time $\vanden$ given by inverse binomial sampling with $\nnum \in \mathbb N$. Let $\ff{x}{i}$ denote $x(x-1) \cdots (x-i+1)$, for $x \in \mathbb R$, $i \in \mathbb N$; and $\ff{x}{0}=1$. The normalized lower incomplete gamma function is defined as
\begin{align}
\gamma(t,u) & = \frac{1}{\Gamma(t)} \int_0^u s^{t-1}\exp(-s) \,\diff s,
\end{align}
and satisfies the following well-known relationship \citep[eq.~(6.5.21)]{Abramowitz70}, which will be used throughout the paper:
\begin{equation}
\label{eq: gamma vecinas}
\gamma(t-1,u) = \gamma(t,u) + \frac{u^{t-1} \exp(-u)}{\Gamma(t)}.
\end{equation}
The random variable $\vanden$ has a negative binomial distribution, with probability function $f_\nnum (\nden) = \Pr[\vanden = \nden]$ given by
$
f_\nnum (\nden)
= \ff{(\nden-1)}{\nnum-1} p^\nnum (1-p)^{\nden-\nnum} / (\nnum-1)!,
$
$\nden \geq \nnum$.
The corresponding distribution function will be denoted as $F_\nnum(\nden)$. Similarly, the probability function of a binomial random variable with parameters $\nden$ and $p$ is denoted as $b_{\nden,p}(i) = \ff{\nden}{i}p^i (1-p)^{\nden-i} /i!$, $0 \leq i \leq \nnum$. For an arbitrary nonrandomized estimator $\hatvap = \fest(\vanden)$ and a loss function $L(\hatvap/p)$, the risk $\risk(p)$ is
\begin{equation}
\risk(p) = \E[L(\hatvap/p)] = \sum_{\nden=\nnum}^\infty f_\nnum (\nden) L(\fest(\nden)/p).
\end{equation}

For $\nnum \geq 2$ and $a,b \geq 0$, the loss functions \eqref{eq: loss abs error gen} and \eqref{eq: loss upper ratio gen} satisfy the sufficient conditions of \citet[theorem 1]{Mendo10_env}, and thus any estimator $\hatvap = \fest(\vanden)$ with $\lim_{\nden \rightarrow \infty} \nden \fest(\nden) = \Omega > 0$ has an asymptotic risk as $p \rightarrow 0$, which can be computed as
\begin{equation}
\label{eq: riskas nu}
\lim_{p \rightarrow 0} \risk(p) = \frac{1}{(\nnum-1)!} \int_0^\infty \nu^{\nnum-1} \exp(-\nu) L(\Omega/\nu) \,\diff \nu.
\end{equation}
In particular, this holds for any estimator that can be expressed as
\begin{equation}
\label{eq: estim Omega d}
\hatvap = \frac{\Omega}{\vanden + d}
\end{equation}
with $\Omega >0$, $d > -\nnum$.

Consider a generic estimator of the form \eqref{eq: estim Omega d}. Denoting $\ndcero = \lfloor \Omega/p - d\rfloor$, the risk associated with the loss function \eqref{eq: loss abs error gen} can be written as
\begin{equation}
\label{eq: abs error gen risk 1}
\begin{split}
\risk(p) & = a \sum_{\nden=\ndcero+1}^{\infty} \left( 1 - \frac{\Omega}{(\nden+d)p} \right) f_\nnum(\nden) + b \sum_{\nden=\nnum}^{\ndcero} \left( \frac{\Omega}{(\nden+d)p} - 1 \right) f_\nnum(\nden) \\
& = -a \sum_{\nden=\nnum}^{\infty} \left( \frac{\Omega}{(\nden+d)p} - 1 \right) f_\nnum(\nden) +
(a+b) \sum_{\nden=\nnum}^{\ndcero} \left( \frac{\Omega}{(\nden+d)p} - 1 \right) f_\nnum(\nden).
\end{split}
\end{equation}
Particularizing to $\Omega=\nnum-1$ and $d=-1$, which yields the uniformly minimum variance unbiased (UMVU) estimator \citep{Mikulski76}, and taking into account the identities \citep{Mendo09a}
\begin{align}
\label{eq: f nnum-1 f nnum}
f_{\nnum-1}(\nden-1) & = \frac{(\nnum-1)f_\nnum(\nden)}{(\nden-1)p} & \text{for } \nnum & \geq 2,\ \nden \geq \nnum, \\
\label{eq: F nnum-1 F nnum b}
F_{\nnum-1}(\nden-1)& = F_\nnum(\nden) + (1-p) b_{\nden-1,p}(\nnum-1) & \text{for } \nnum & \geq 2,\ \nden \geq \nnum,
\end{align}
it is seen that, for $\nnum \geq 2$, the first summand in \eqref{eq: abs error gen risk 1} becomes $0$, and
\begin{equation}
\label{eq: abs error gen risk 2}
\begin{split}
\risk(p) & = (a+b) \sum_{\nden=\nnum}^{\ndcero} ( f_{\nnum-1}(\nden-1) - f_\nnum(\nden) ) = (a+b) (F_{\nnum-1}(\ndcero-1) - F_{\nnum}(\ndcero)) \\
& = (a+b) (1-p) b_{\ndcero-1,p}(\nnum-1).
\end{split}
\end{equation}
The case $a=b=1$ is analyzed in \citet{Mendo09a}. Comparing \eqref{eq: abs error gen risk 2} with \citet[eq.~(12)]{Mendo09a}, the expression of the risk for $a,b$ arbitrary is seen to be a straightforward generalization of that for $a=b=1$. As a consequence, the following result holds.

\begin{theorem}
\label{teo: loss abs error gen}
Consider the loss function given by \eqref{eq: loss abs error gen} with $a,b \geq 0$, $(a,b) \neq (0,0)$. For $\nnum \geq 2$, the risk $\risk(p)$ associated with the estimator $\hatvap = (\nnum-1)/(\vanden-1)$ satisfies
\begin{equation}
\label{eq: teo: loss abs error gen}
\risk(p) < \lim_{p \rightarrow 0} \risk(p) \quad \text{for any } p \in (0,1),
\end{equation}
with
\begin{equation}
\label{eq: teo: loss abs error gen, riskas}
\lim_{p \rightarrow 0} \risk(p) = \frac{(a+b) (\nnum-1)^{\nnum-2} \exp(-\nnum+1)} {(\nnum-2)!}.
\end{equation}
\end{theorem}

In addition, as will be seen in Section~\ref{parte: disc}, under certain conditions this estimator approaches the asymptotically optimum estimator discussed in \citet{Mendo10_env}.

For the loss function \eqref{eq: loss upper ratio gen}, the risk associated with an estimator of the form \eqref{eq: estim Omega d} can be decomposed in a similar way as for \eqref{eq: loss abs error gen}. Namely, $\risk(p) = \risk_1(p)+\risk_2(p)$ with
\begin{align}
\label{eq: upper ratio gen risk1 1}
\risk_1(p) & = a \sum_{\nden=\ndcero+1}^{\infty} \left( \frac{(\nden+d)p}{\Omega} - 1 \right) f_\nnum(\nden), \\
\label{eq: upper ratio gen risk2 1}
\risk_2(p) & = b \sum_{\nden=\nnum}^{\ndcero} \left( \frac{\Omega}{(\nden+d)p} - 1 \right) f_\nnum(\nden).
\end{align}
Assuming $d \leq 0$ in \eqref{eq: upper ratio gen risk1 1} and taking into account that, as per \eqref{eq: f nnum-1 f nnum}, $\nden p f_{\nnum}(\nden) = \nnum f_{\nnum+1}(\nden+1)$, it follows that
\begin{equation}
\label{eq: upper ratio gen risk1 2}
\risk_1(p) \leq a \sum_{\nden=\ndcero+1}^{\infty} \left( \frac{\nden p}{\Omega} - 1 \right) f_\nnum(\nden)
= \frac{a\nnum}{\Omega} (1-F_{\nnum+1}(\ndcero+1)) + a(F_\nnum(\ndcero)-1) ,
\end{equation}
with strict inequality if $d<0$. As for $\risk_2(p)$, assuming $d \geq -1$, it stems from \eqref{eq: f nnum-1 f nnum} and \eqref{eq: upper ratio gen risk2 1} that
\begin{equation}
\label{eq: upper ratio gen risk2 2}
\risk_2(p) \leq b \sum_{\nden=\nnum}^{\ndcero} \left( \frac{\Omega}{(\nden-1)p} - 1 \right) f_\nnum(\nden)
= \frac{b\Omega}{\nnum-1} F_{\nnum-1}(\ndcero-1) - b F_\nnum(\ndcero),
\end{equation}
with strict inequality if $d > -1$ and $\ndcero \geq \nnum$. As a result of \eqref{eq: upper ratio gen risk1 2} and \eqref{eq: upper ratio gen risk2 2}, for any $d \in [-1,0]$ the risk satisfies
\begin{equation}
\begin{split}
\label{eq: risk 1}
\risk(p) \leq \frac{b\Omega}{\nnum-1} F_{\nnum-1}(\ndcero-1) + (a-b) F_\nnum(\ndcero) - \frac{a\nnum}{\Omega} F_{\nnum+1}(\ndcero+1) + a\left(\frac{\nnum}{\Omega} - 1 \right).
\end{split}
\end{equation}
The right-hand side of \eqref{eq: risk 1} is greatly simplified if $\Omega$ is chosen as any value $\Omegail>0$ such that
\begin{equation}
\label{eq: cond Omega}
\frac{a\nnum}{\Omegail} - \frac{b\Omegail}{\nnum-1} = a-b,
\end{equation}
for in that case, applying the identity \eqref{eq: F nnum-1 F nnum b},
\begin{equation}
\begin{split}
\label{eq: risk 2}
\risk(p) & \leq \frac{b\Omegail}{\nnum-1} (F_{\nnum-1}(\ndcero-1) - F_\nnum(\ndcero)) + \frac{a\nnum}{\Omegail} (F_\nnum(\ndcero)-F_{\nnum+1}(\ndcero+1)) + a\left(\frac{\nnum}{\Omegail} - 1 \right) \\
& = \frac{b\Omegail}{\nnum-1} (1-p) b_{\ndcero-1,p}(\nnum-1) + \frac{a\nnum}{\Omegail} (1-p) b_{\ndcero,p}(\nnum) + a\left(\frac{\nnum}{\Omegail} - 1 \right).
\end{split}
\end{equation}
The advantage of this expression is that the terms $(1-p) b_{\ndcero-1,p}(\nnum-1)$ and $(1-p) b_{\ndcero,p}(\nnum)$ lend themselves to analysis more easily than the distribution functions in \eqref{eq: risk 1}.

The condition \eqref{eq: cond Omega} on $\Omegail$ has a single positive solution for $a,b \geq 0$, $(a,b) \neq (0,0)$, namely
\begin{equation}
\label{eq: Omega il}
\Omegail =
\begin{cases}
(\nnum-1) \left( 1 + \frac{a+b}{2b} \left( \sqrt{1 + \frac{4ab}{(\nnum-1)(a+b)^2}} - 1\right) \right) & \text{for } a,b >0,\\
\nnum-1 & \text{for } a=0,\ b>0, \\
\nnum & \text{for } b=0,\ a>0.
\end{cases}
\end{equation}
It is easily seen that this reduces to $\sqrt{\nnum(\nnum-1)}$ for $a=b > 0$. In addition, the following holds.

\begin{proposition}
\label{prop: Omega invlin cota}
The value of $\Omegail$ given by \eqref{eq: Omega il} lies in the interval $(\nnum-1, \nnum)$ for $a,b > 0$.
\end{proposition}

As a consequence of Proposition~\ref{prop: Omega invlin cota}, for any $a,b \geq 0$, $(a,b) \neq (0,0)$, the value $\Omegail$ defined by \eqref{eq: Omega il} satisfies $\Omegail \in [\nnum-1, \nnum]$ . Taking into account that $\hatvapUMVU = (\nnum-1)/(\vanden-1)$ is the UMVU estimator and that $\hatvapML = \nnum/\vanden$ is the maximum likelihood (ML) estimator \citep{Best74}, the estimator $\hatvap$ given by \eqref{eq: estim Omega d} with $\Omega \in [\nnum-1, \nnum]$ and $d \in [-1,0]$ is seen to be a ``reasonable'' one, in the sense that it is ``close'' to the UMVU and ML estimators. As will be seen in Section~\ref{parte: disc}, in certain cases the proposed estimator is also close to the asymptotically optimum estimator in the sense of \citet{Mendo10_env}.

The preceding arguments justify that the estimator given by \eqref{eq: estim Omega d} with $\Omega \in [\nnum-1, \nnum]$ and $d \in [-1,0]$ is worth considering. In fact, for adequate choices of $\Omega$ and $d$, it satisfies the important property that the risk is guaranteed not to exceed its asymptotic value, as established by the next theorem.

\begin{theorem}
\label{teo: loss upper ratio gen}
Consider the loss function given by \eqref{eq: loss upper ratio gen} with $a,b \geq 0$, $(a,b) \neq (0,0)$. For $\nnum \geq 2$, the estimator $\hatvap = \Omegail/\vanden$ with $\Omegail$ given by \eqref{eq: Omega il} satisfies
\begin{equation}
\label{eq: teo: loss upper ratio gen}
\risk(p) < \lim_{p \rightarrow 0} \risk(p) \quad \text{for any } p \in (0,1),
\end{equation}
with
\begin{equation}
\label{eq: teo: loss upper ratio gen, riskas}
\lim_{p \rightarrow 0} \risk(p) = a \left( \frac{\nnum}{\Omegail} - 1 \right) + \left( \frac{a\nnum}{\Omegail} + b \right) \frac{\Omegail^{\nnum-1} \exp(-\Omegail)}{(\nnum-1)!}.
\end{equation}
\end{theorem}

\section{Discussion and additional properties}
\label{parte: disc}

\subsection{Significance of the results}

It has  been shown in Section~\ref{parte: res} that similar results to those already known for mean absolute error, mean squared error and confidence level also hold for generalized mean absolute error (Theorem~\ref{teo: loss abs error gen}) and generalized mean symmetric ratio (Theorem~\ref{teo: loss upper ratio gen}). Specifically, it has been proved that, for the proposed estimators, $\sup_{p \in (0,1)} \risk(p) = \lim_{p \rightarrow 0} \risk(p)$. In the following, $\bar\risk$ will denote the value of $\lim_{p \rightarrow 0} \risk(p)$, or equivalently $\sup_{p \in (0,1)} \risk(p)$, for the estimators in Theorems~\ref{teo: loss abs error gen} and \ref{teo: loss upper ratio gen}.

The importance of these results lies in the fact that no knowledge is required about $p$. Thus, given any desired value $\riskej$ for the risk, an adequate $\nnum$ can be selected such that the risk is guaranteed not to exceed $\riskej$, irrespective of $p$. Namely, it suffices to choose $\nnum$ as the minimum value for which $\bar\risk$, computed from \eqref{eq: teo: loss abs error gen, riskas} or from \eqref{eq: teo: loss upper ratio gen, riskas}, is less than or equal to $\riskej$. As an illustration, Figure~\ref{fig: barrisk_N_il} depicts $\bar\risk$ as a function of $\nnum$ for the loss given by \eqref{eq: loss upper ratio gen} with $a=b=1$.
It is seen, for example, that $\nnum=75$ suffices to guarantee a risk lower than $0.1$, that is, a mean symmetric ratio lower than $1.1$.

\begin{figure}%
\centering%
\includegraphics[width = .65\textwidth]{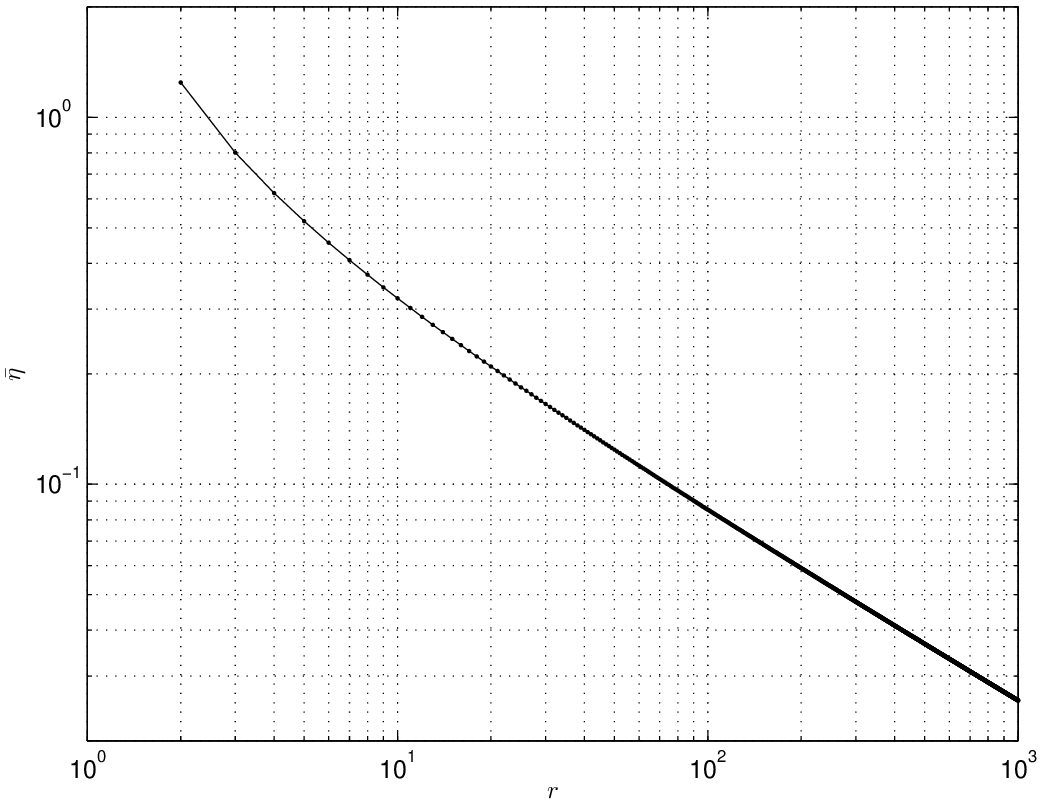}%
\caption{Risk guaranteed not to be exceeded for inverse-linear loss \eqref{eq: loss upper ratio gen} with $a=b=1$}%
\label{fig: barrisk_N_il}%
\end{figure}%

\subsection{Comparison with minimax estimators}

The presented results are valid for specific estimators, given by \eqref{eq: estim Omega d} with certain fixed values for $\Omega$ and $d$. It is natural to ask to what extent the results could be improved by considering other estimators, i.e.~how much lower risks could be guaranteed not to be exceeded (or equivalently how much $\sup_{p \in (0,1)} \risk(p)$ could be reduced). By definition, an estimator that is optimum according to this criterion (i.e.~which minimizes $\sup_{p \in (0,1)} \risk(p)$ over all possible estimators), if it exists, is a \emphmio{minimax} estimator.

This question can be addressed on the basis of the analysis in \citet{Mendo10_env}. For $a, b > 0$, both \eqref{eq: loss abs error gen} and \eqref{eq: loss upper ratio gen} satisfy the assumptions of \citet[theorem 3]{Mendo10_env}. This implies that there exists a value of $\Omega$, denoted as $\Omega^*$, such that any estimator $\hatvap = \fest(\vanden)$ with $\lim_{\nden \rightarrow \infty} \nden \fest(\nden) = \Omega^*$ minimizes $\limsup_{p \rightarrow 0} \risk(p)$ over all estimators, including randomized ones. Thus any such estimator is asymptotically optimum, in the sense of achieving the minimum possible $\limsup_{p \rightarrow 0} \risk(p)$. This minimum, which will be denoted as $\risk^*$, restricts the values $\riskej$ that the risk can be guaranteed not to exceed for $p$ arbitrary. Namely, if an estimator guarantees that $\risk(p) \leq \riskej$ for a given $\riskej$, then necessarily $\riskej \geq \risk^*$. As a consequence, the risk $\bar\risk$ that is guaranteed not to be exceeded by the specific estimators considered in Section~\ref{parte: res} is at most $\bar\risk/\risk^*$ times larger than what could be achieved by a minimax estimator.

The value $\risk^*$ is obtained as follows. Consider the loss function \eqref{eq: loss abs error gen} first. For $\Omega$ arbitrary, \eqref{eq: riskas nu} gives
\begin{equation}
\label{eq: bar risk ll}
\lim_{p \rightarrow 0} \risk(p) = \frac{(a+b)\Omega \gamma(\nnum-1,\Omega)}{\nnum-1} - (a+b) \gamma(\nnum,\Omega) + a \left( 1 - \frac{\Omega}{\nnum-1}\right).
\end{equation}
Its derivative
\begin{equation}
\label{eq: diff bar risk diff Omega ll}
\frac{\diff}{\diff\Omega} \lim_{p \rightarrow 0} \risk(p) = \frac{(a+b)\gamma(\nnum-1,\Omega) - a}{\nnum-1}
\end{equation}
is seen to be monotone increasing. Therefore the minimizing value $\Omega^*$ is unique, and is determined by the condition $\diff(\lim_{p \rightarrow 0} \risk(p)) /\diff\Omega = 0$, that is,
\begin{equation}
\label{eq: Omega opt ll}
\gamma(\nnum-1,\Omega^*) = \frac{a}{a+b}.
\end{equation}
Setting $\Omega=\Omega^*$ in \eqref{eq: bar risk ll}, substituting \eqref{eq: Omega opt ll} and making use of \eqref{eq: gamma vecinas},
\begin{equation}
\label{eq: risk opt ll}
\risk^* = \frac{(a+b) {\Omega^*}^{\nnum-1} \exp(-\Omega^*)} {(\nnum-1)!}.
\end{equation}
The value $\Omega^*$ can be computed numerically from \eqref{eq: Omega opt ll}, and $\risk^*$ is then obtained by means of \eqref{eq: risk opt ll}.

Regarding the loss function \eqref{eq: loss upper ratio gen}, for $\Omega$ arbitrary \eqref{eq: riskas nu} gives
\begin{equation}
\label{eq: bar risk il}
\lim_{p \rightarrow 0} \risk(p) = \frac{b\Omega \gamma(\nnum-1,\Omega)}{\nnum-1} + (a-b) \gamma(\nnum,\Omega) - \frac{a\nnum \gamma(\nnum+1,\Omega)}{\Omega} + a \left(\frac \nnum \Omega - 1 \right).
\end{equation}
Again, it is easily seen that
\begin{equation}
\label{eq: diff bar risk diff Omega il}
\frac{\diff}{\diff\Omega} \lim_{p \rightarrow 0} \risk(p) =
\frac{b\gamma(\nnum-1,\Omega)}{\nnum-1}
- \frac{a\nnum (1-\gamma(\nnum+1,\Omega))}{\Omega^2}
\end{equation}
is monotone increasing, and thus there is a single minimizing value $\Omega^*$, which satisfies
\begin{equation}
\label{eq: Omega opt il}
\frac{{\Omega^*}^2}{\nnum(\nnum-1)} = \frac{a(1-\gamma(\nnum+1,\Omega^*))}{b\gamma(\nnum-1,\Omega^*)}.
\end{equation}
From \eqref{eq: gamma vecinas}, \eqref{eq: bar risk il} and \eqref{eq: Omega opt il},
\begin{equation}
\label{eq: risk opt il}
\risk^* = \left(a+\left( \frac{2\Omega^*}{\nnum-1}-1\right)b \right) \gamma(\nnum-1,\Omega^*) + \frac{(b-a){\Omega^*}^{\nnum-1} \exp(-\Omega^*)}{(\nnum-1)!} - a.
\end{equation}
The expressions \eqref{eq: Omega opt il} and \eqref{eq: risk opt il} allow numerically computing $\risk^*$.

Figure~\ref{fig: rel_opt} shows, for the loss functions and estimators considered in Theorems~\ref{teo: loss abs error gen} and \ref{teo: loss upper ratio gen}, the degradation factor $\bar\risk/\risk^*$ as a function of $a/b$, with $\nnum$ as a parameter. As is seen, for $a/b$ not too far from $1$ the degradation factor is close to $1$, that is, the considered estimators are nearly optimum. Furthermore, there is a value of $a/b$ for which each estimator is precisely optimum, i.e.~minimax, as established by the following.

\begin{figure}%
\centering%
\subfigure[Linear-linear loss \eqref{eq: loss abs error gen}]{%
\label{fig: rel_opt_ll}%
\includegraphics[width = .5\textwidth]{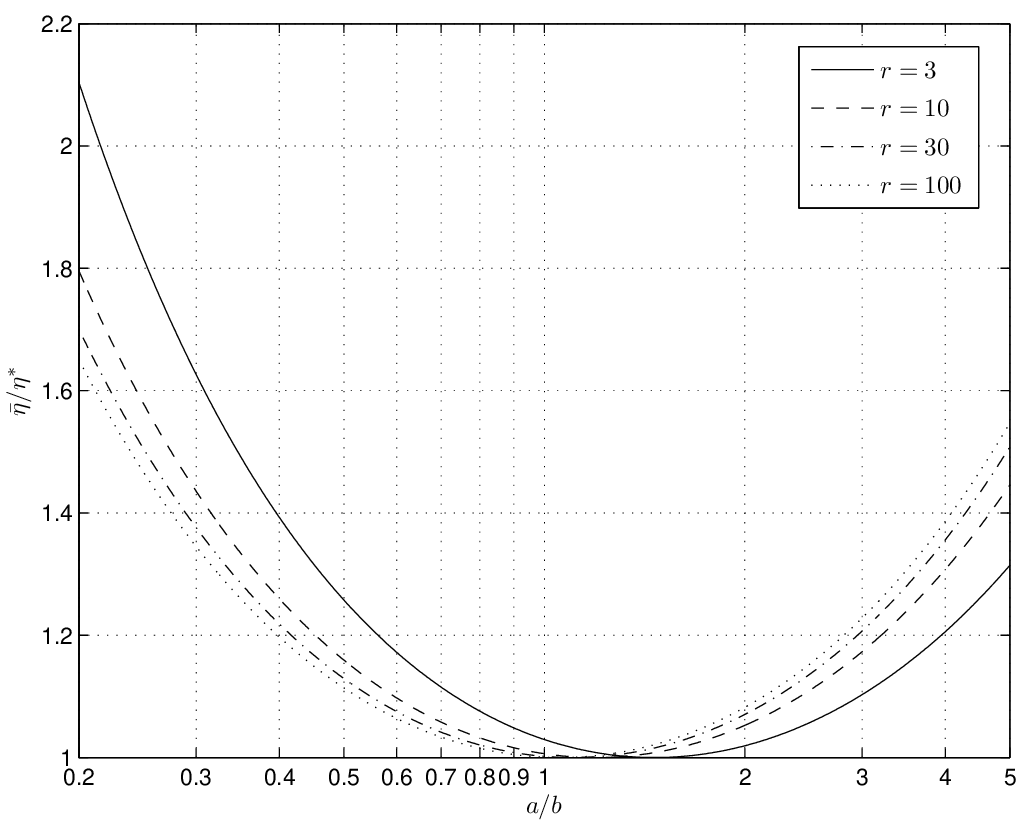}%
}%
\subfigure[Inverse-linear loss \eqref{eq: loss upper ratio gen}]{%
\label{fig: rel_opt_il}%
\includegraphics[width = .5\textwidth]{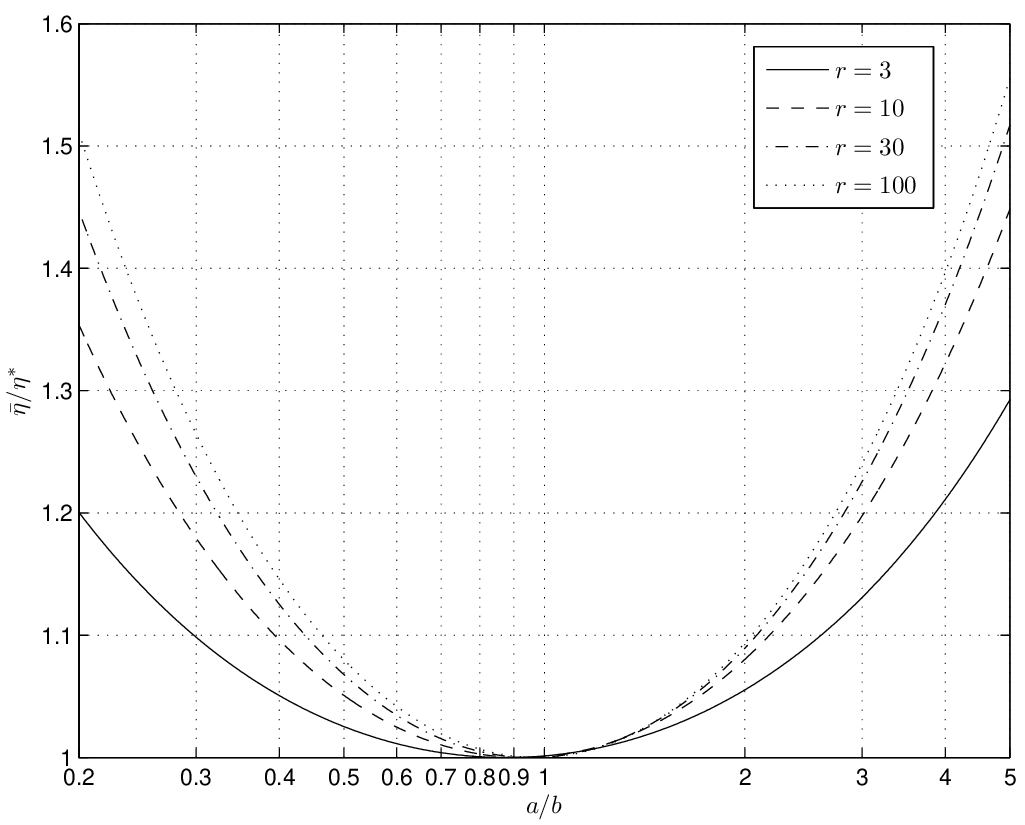}%
}%
\caption{Degradation factor $\bar\risk/\risk^*$ as a function of $a/b$ and $\nnum$}%
\label{fig: rel_opt}%
\end{figure}%

\begin{proposition}
\label{prop: a b para minimax}
For each of the loss functions \eqref{eq: loss abs error gen} and \eqref{eq: loss upper ratio gen}, there exists a unique value of the ratio $a/b$ for which the estimator considered in Theorem~\ref{teo: loss abs error gen} or \ref{teo: loss upper ratio gen} respectively is minimax, that is, minimizes $\sup_{p \in (0,1)} \risk(p)$ over all (possibly randomized) estimators. For the loss function \eqref{eq: loss abs error gen} this value is given by
\begin{equation}
\label{eq: a b minimax ll}
\frac a b = \frac{\gamma(\nnum-1,\nnum-1)}{1-\gamma(\nnum-1,\nnum-1)},
\end{equation}
and for \eqref{eq: loss upper ratio gen} it is determined by the condition
\begin{equation}
\label{eq: a b minimax il}
\frac{\gamma(\nnum-1,\Omegail)}{1-\gamma(\nnum+1,\Omegail)} = \frac{\nnum(\Omegail-\nnum+1)}{\Omegail(\nnum-\Omegail)}
\end{equation}
with $\Omegail$ as in \eqref{eq: Omega il}.
\end{proposition}

\subsection{Minimaxity for asymptotically large $\nnum$ in the case $a=b$}

The specific values of the ratio $a/b$ determined by Proposition~\ref{prop: a b para minimax} can be shown to tend to $1$ as $\nnum \rightarrow \infty$. Related to this, the following establishes that for $a=b$ the proposed estimators are asymptotically minimax as $\nnum \rightarrow \infty$.

\begin{proposition}
\label{prop: a b 1 asint minimax}
For the loss functions \eqref{eq: loss abs error gen} and \eqref{eq: loss upper ratio gen} with $a=b$, each of the estimators considered in Theorems~\ref{teo: loss abs error gen} and \ref{teo: loss upper ratio gen}, respectively, approaches a minimax estimator asymptotically as $\nnum \rightarrow \infty$, in the sense that $\lim_{\nnum \rightarrow \infty} \bar\risk/\risk^* = 1$.
\end{proposition}

As a consequence of this result, for $a=b$ and large $\nnum$ the considered estimators are approximately optimum in the minimax sense. This is illustrated in Figure~\ref{fig: rel_opt_N}, which shows the degradation factor $\bar\risk/\risk^*$ as a function of $\nnum$. In fact, $\bar\risk/\risk^*$ is seen to be very low even for small $\nnum$, and in particular for the range of values of $\nnum$ that are commonly used in practice. Thus, for example, the mean absolute error (loss function \eqref{eq: loss abs error gen} with $a=b=1$) that is guaranteed not to be exceeded according to Theorem~\ref{teo: loss abs error gen} is within $1\%$ of the minimax mean absolute error for $7 \leq \nnum \leq 1000$. Similarly, defining risk as mean symmetric ratio minus $1$ (loss function \eqref{eq: loss upper ratio gen} with $a=b=1$), the risk that is guaranteed not to be exceeded as per Theorem~\ref{teo: loss upper ratio gen} is within $0.1\%$ of the minimax risk for the same range of values of $\nnum$.

\begin{figure}%
\centering%
\subfigure[Linear-linear loss \eqref{eq: loss abs error gen}]{%
\label{fig: rel_opt_N_ll}%
\includegraphics[width = .499\textwidth]{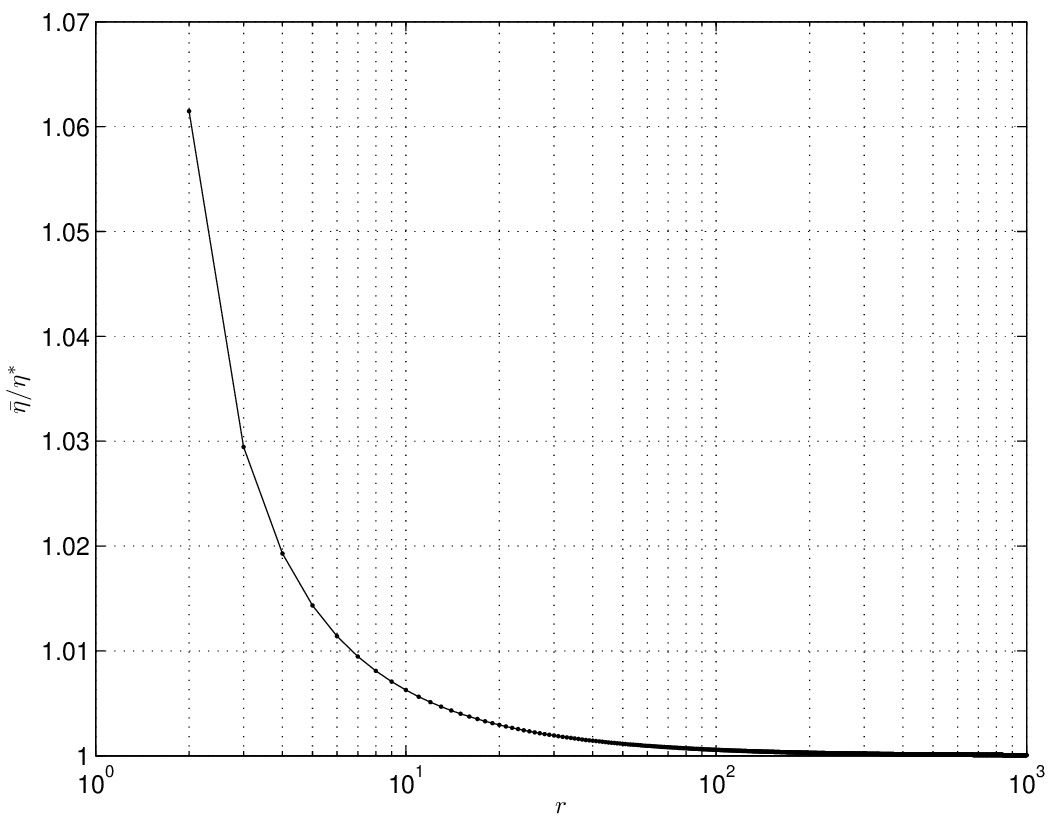}%
}%
\subfigure[Inverse-linear loss \eqref{eq: loss upper ratio gen}]{%
\label{fig: rel_opt_N_il}%
\includegraphics[width = .511\textwidth]{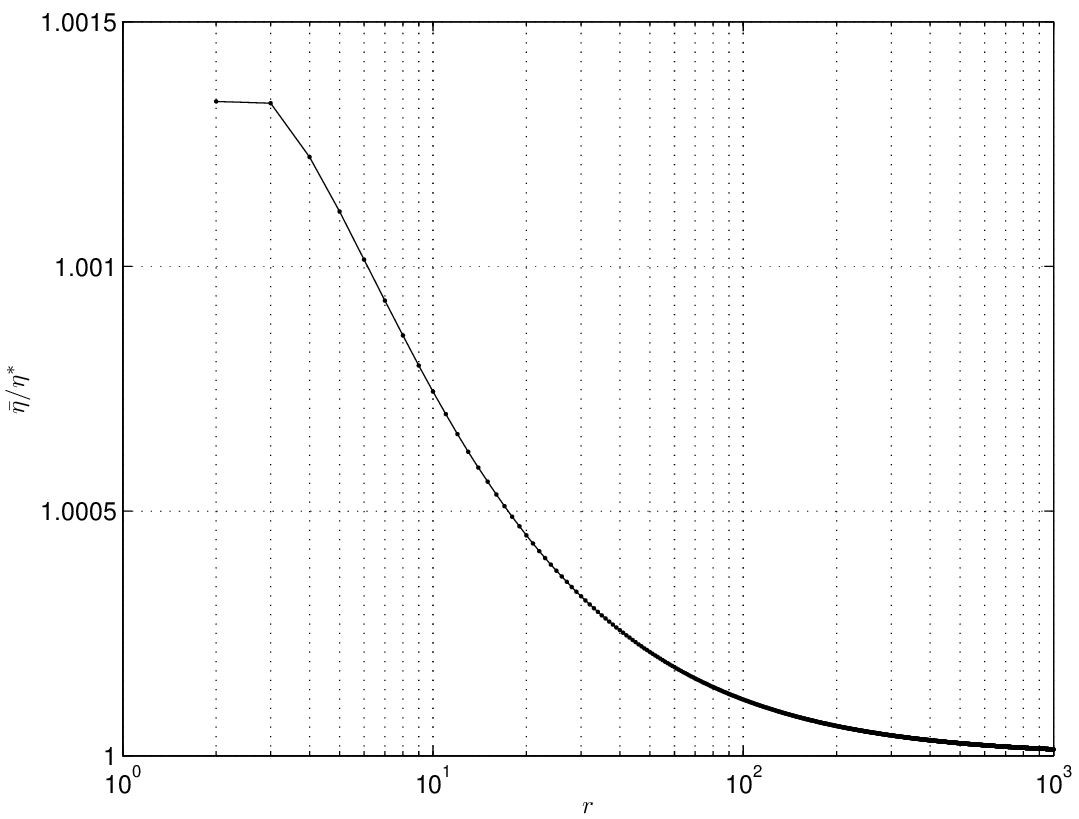}%
}%
\caption{Degradation factor $\bar\risk/\risk^*$ for $a=b$ as a function of $\nnum$}%
\label{fig: rel_opt_N}%
\end{figure}%

\section{Proofs}
\label{parte: proofs}

For $\rho \in \mathbb N$, $\rho \geq 2$;
$\mu \geq \rho-1$; and $t \in (0,1)$, let $Y_\rho(\mu,t)$ be defined as
\begin{equation}
\label{eq: Y nnum gen}
Y_\rho(\mu,t) =
\frac{ \ff{(\mu-1)}{\rho-1} t^{\rho-1} (1-t)^{\mu-\rho+1} }{(\rho-1)!}.
\end{equation}

\begin{proof}[Proof of Theorem~\ref{teo: loss abs error gen}]
The result immediately stems from \eqref{eq: abs error gen risk 2} and the analysis in \citet{Mendo09a}.
\end{proof}

\begin{lemma}
\label{lem: suma int}
The following inequality holds for $r \geq 2$, $d \in [-1, 0]$, $j \geq 2$.
\begin{equation}
\label{eq: suma int}
\sum_{i=1}^{\nnum-1} (i+d+1)^j > \frac{(\nnum+d)^{j+1}}{j+1}.
\end{equation}
\end{lemma}

\begin{proof}
The sum in \eqref{eq: suma int} can be expressed as the area covered by the $\nnum-1$ rectangles of width $1$ and height $(i+d+1)^j$, $i=1,\ldots,\nnum-1$ in Figure~\ref{fig: fig-dem}, or equivalently as the shaded area comprised by $\nnum-2$ unit-width trapezoids plus two half-width rectangles. Since the curve $(x+d+3/2)^j$ touches the upper vertices of the trapezoids and is convex, the following inequality can be written:
\begin{equation}
\label{eq: suma int larga}
\begin{split}
\sum_{i=1}^{\nnum-1} (i+d+1)^j & \geq \int_{1/2}^{\nnum-3/2} \left (x+d+ \frac 3 2 \right)^j \diff x + \frac{(d+2)^j}{2} + \frac{(\nnum+d)^j}{2} \\
& = \frac{(\nnum+d)^{j+1}}{j+1} + \frac{(r+d)^j}{2} + (d+2)^j \left( \frac 1 2 - \frac{d+2}{j+1} \right).
\end{split}
\end{equation}
For $j \geq 3$ the term $1/2 - (d+2)/(j+1)$ in~\eqref{eq: suma int larga} is nonnegative, which ensures that \eqref{eq: suma int} holds. For $j=2$, \eqref{eq: suma int larga} reduces to
\begin{equation}
\label{eq: suma int j 2}
\sum_{i=1}^{\nnum-1} (i+d+1)^j \geq \frac{(\nnum+d)^{3}}{3} + \frac{-2d^3-6d^2+6(\nnum-2)d+3\nnum^2-4}{6}.
\end{equation}
The second summand in \eqref{eq: suma int j 2} has a derivative with respect to $d$ equal to $-d^2-2d+\nnum-2$, which is nonnegative for $d \in [-1,0]$, $\nnum \geq 2$. Thus this summand is lower bounded by its value at $d=-1$, i.e.~$(3\nnum^2-6\nnum+4)/6$, which is positive. Therefore \eqref{eq: suma int} also holds for $j=2$.
\begin{figure}%
\centering%
\includegraphics[width = .65\textwidth, clip]{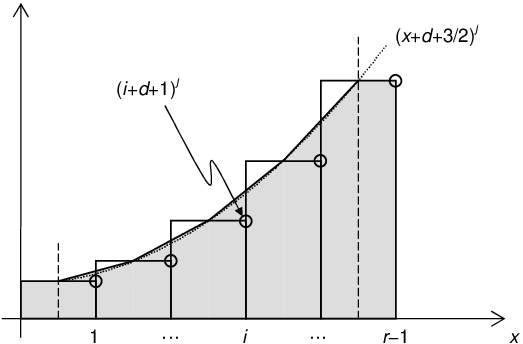}%
\caption{Illustration of \eqref{eq: suma int larga}}%
\label{fig: fig-dem}%
\end{figure}%
\end{proof}

\begin{lemma}
\label{lem: y nnum Taylor}
Given $\rho$, $\mu$, $\Omega$ and $\delta$ such that (i) $\rho \in \mathbb N, \rho \geq 2$; (ii) $\Omega \in [\rho-1,\rho]$; (iii) $\delta \in [-1,0]$; (iv) $\mu > \rho-1$; and (v) $\mu > \Omega-\delta-1$, the following hold:
\begin{enumerate}
\item
\label{lem: y nnum Taylor a}
$Y_\rho(\mu,\Omega/(\mu+\delta+1))$ is a strictly increasing function of $\mu$, with
\begin{equation}
\label{eq: lim y nnum a}
\lim_{\mu \rightarrow \infty} Y_\rho(\mu,\Omega/(\mu+\delta+1)) = {\Omega^{\rho-1} \exp(-\Omega)} / {(\rho-1)!}.
\end{equation}
\item
\label{lem: y nnum Taylor b}
$Y_{\rho+1}(\mu+1,\Omega/(\mu+\delta+1))$ is a strictly increasing function of $\mu$, with
\begin{equation}
\label{eq: lim y nnum b}
\lim_{\mu \rightarrow \infty} Y_{\rho+1}(\mu+1,\Omega/(\mu+\delta+1)) = {\Omega^\rho \exp(-\Omega)} / {\rho!}.
\end{equation}
\end{enumerate}
\end{lemma}

\begin{proof}
According to hypotheses (iv) and (v), it holds that $\mu > \rho-1$ and $\Omega/(\mu+\delta+1) < 1$, and thus $Y_\rho(\mu,\Omega/(\mu+\delta+1))$ and $Y_{\rho+1}(\mu+1,\Omega/(\mu+\delta+1))$ are well defined from \eqref{eq: Y nnum gen}.

The proof will be carried out separately for parts \ref{lem: y nnum Taylor a} and \ref{lem: y nnum Taylor b} of the Lemma.

\ref{lem: y nnum Taylor a}
It is convenient to make the change of variable $t=\Omega/(\mu+\delta+1)$, by which $Y_\rho(\mu,\Omega/(\mu+\delta+1))$ is expressed as $Y_\rho(\Omega/t-\delta-1,t)$. It will be shown that
\begin{equation}
\label{eq: lim y nnum a cambio de var}
\lim_{t \rightarrow 0} Y_\rho ( \Omega/t-\delta-1, t ) =  {\Omega^{\rho-1} \exp(-\Omega)} / {(\rho-1)!},
\end{equation}
which is equivalent to \eqref{eq: lim y nnum a}; and that $Y_\rho(\Omega/t-\delta-1,t)$ is a strictly decreasing function of $t$, which will imply that $Y_\rho(\mu,\Omega/(\mu+\delta+1))$ strictly increases with $\mu$. From \eqref{eq: Y nnum gen},
\begin{equation}
\label{eq: log Y nnum gen}
\log Y_\rho(\mu,t) = \sum_{i=1}^{\rho-1} \log \frac{(\mu-i)t}{\Omega} + \sum_{i=1}^{\rho-1} \log \frac{\Omega}{i} +    (\mu-\rho+1)\log(1-t),
\end{equation}
and thus
\begin{equation}
\label{eq: log Y nnum gen sust}
\begin{split}
\log Y_\rho \left( \frac{\Omega}{t}-\delta-1,t \right)
& = \sum_{i=1}^{\rho-1} \log \left( 1 - \frac{(i+\delta+1)t}{\Omega} \right) \\
& \quad + \sum_{i=1}^{\rho-1} \log \frac{\Omega}{i} + \left( \frac{\Omega}{t}-\rho-\delta \right)\log(1-t).
\end{split}
\end{equation}
Taking into account that $(\rho+\delta)t/\Omega = (\rho+\delta)/(\mu+\delta+1) < 1$ as a result of (iv), and that $t<1$, the Taylor expansion $\log(1-t) = -\sum_{j=1}^\infty t^j/j$, $|t|<1$ can be used in \eqref{eq: log Y nnum gen sust} to yield
\begin{align}
\label{eq: serie c j}
\log Y_\rho \left( \frac{\Omega}{t}-\delta-1,t \right) & = -\sum_{j=0}^\infty c_{j} t^j, \\
\label{eq: c 0}
c_{0} & = \Omega + \sum_{i=1}^{\rho-1} \log \frac i \Omega, \\
\label{eq: c j}
c_{j} & = \frac{\Omega}{j+1} - \frac{\rho+\delta}{j} + \frac{1}{j\Omega^j} \sum_{i=1}^{\rho-1} (i+\delta+1)^j \quad \text{for } j \geq 1.
\end{align}
The equalities \eqref{eq: serie c j} and \eqref{eq: c 0} imply \eqref{eq: lim y nnum a cambio de var}, and thus \eqref{eq: lim y nnum a}.

To prove that $Y_\rho(\Omega/t-\delta-1,t)$ strictly decreases with $t$, it suffices to show that the coefficients $c_{j}$ satisfy $c_{j} \geq 0$, $j \geq 1$, with strict inequality for some $j$. For $j \geq 2$, \eqref{eq: c j} and Lemma~\ref{lem: suma int} yield
\begin{equation}
\begin{split}
\label{eq: c j 2}
j(j+1) c_{j} & > j\Omega -(j+1)(\rho+\delta) + (\rho+\delta) \left(\frac{\rho+\delta}{\Omega}\right)^j \\
& = -j(\rho+\delta-\Omega) + (\rho+\delta) \left( \left(\frac{\rho+\delta}{\Omega}\right)^j -1 \right).
\end{split}
\end{equation}
Taking into account that $\rho+\delta \geq 0$ by hypothesis (iii), and using the inequality
\begin{equation}
\left(\frac{\rho+\delta}{\Omega}\right)^j = \left(1 + \frac{\rho+\delta-\Omega}{\Omega}\right)^j \geq 1+ \frac{j(\rho+\delta-\Omega)}{\Omega},
\end{equation}
it follows from \eqref{eq: c j 2} that
\begin{equation}
\label{eq: c j 3}
c_{j} > \frac{(\rho+\delta-\Omega)^2}{(j+1) \Omega} \geq 0.
\end{equation}
For $j=1$, \eqref{eq: c j} gives
\begin{equation}
\label{eq: c 1}
c_{1}
=  \frac{\Omega^2 - 2\rho\Omega + (\rho-1)(\rho+2)}{2\Omega} + \frac{\delta(\rho-\Omega-1)}{\Omega}.
\end{equation}
Consider the first summand in \eqref{eq: c 1}. The minimum of its numerator with respect to $\Omega$ is attained at $\Omega = \rho$ and equals $\rho-2$. Thus, according to hypothesis (i), this summand is nonnegative. By  (ii) and (iii), the second summand is also nonnegative; and therefore $c_{1} \geq 0$. Consequently $Y_\rho(\Omega/t-\delta-1,t)$ strictly decreases with $t$, and thus $Y_\rho(\mu,\Omega/(\mu+\delta+1))$ strictly increases with $\mu$.

\ref{lem: y nnum Taylor b}
Making the same change of variable as in part \ref{lem: y nnum Taylor a}, and taking into account \eqref{eq: serie c j}--\eqref{eq: c j},
\begin{equation}
\label{eq: serie c' j}
\log Y_{\rho+1} \left( \mu+1, \frac{\Omega}{\mu+\delta+1} \right) = \log Y_{\rho+1} \left( \frac{\Omega}{t}-\delta,t \right) = -\sum_{j=0}^\infty c'_j t^j,
\end{equation}
\begin{align}
\label{eq: c' 0}
c'_{0} & = \Omega + \sum_{i=1}^\rho \log \frac i \Omega, \\
\label{eq: c' j}
c'_{j} &
= \frac{\Omega}{j+1} - \frac{\rho+\delta}{j} + \frac{1}{j\Omega^j} \sum_{i=0}^{\rho-1} (i+\delta+1)^j \geq c_j
\quad \text{for } j \geq 1.
\end{align}
From \eqref{eq: c' j} it follows that $c'_1 \geq 0$ and $c'_j > 0$ for $j \geq 2$. Together with \eqref{eq: serie c' j} and \eqref{eq: c' 0}, this establishes part \ref{lem: y nnum Taylor b} of the Lemma.
\end{proof}

\begin{proof}[Proof of Theorem~\ref{teo: loss upper ratio gen}]
As the case $a=0$, $b>0$ is already covered by Theorem~\ref{teo: loss abs error gen}, it will be assumed that $a>0$. This implies, according to Proposition~\ref{prop: Omega invlin cota}, that $\Omegail > \nnum-1$.

The equality \eqref{eq: teo: loss upper ratio gen, riskas} is obtained substituting the loss function \eqref{eq: loss upper ratio gen} into \eqref{eq: riskas nu} with $\Omega = \Omegail$, and making use of \eqref{eq: gamma vecinas} and \eqref{eq: cond Omega}.

The inequality \eqref{eq: risk 2} can be expressed as
\begin{align}
\label{eq: il risk  1}
\risk(p) & \leq \frac{b\Omegail Y_\nnum(\ndcero,p)}{\nnum-1}  + \frac{a\nnum Y_{\nnum+1}(\ndcero+1,p)}{\Omegail} + a\left(\frac{\nnum}{\Omegail} - 1 \right), \\
\label{eq: ndcero Omega il}
\ndcero & = \lfloor \Omegail/p \rfloor.
\end{align}
From \eqref{eq: ndcero Omega il} it stems that $\ndcero \geq \nnum-1$. Each value of $\ndcero$ has an associated interval $I_\ndcero \subseteq (0,1)$ such that \eqref{eq: ndcero Omega il} holds if and only if $p \in I_\ndcero$. Namely, $I_\ndcero = (\pinf, \psup]$ with $\pinf = \Omegail/(\ndcero+1)$, $\psup = \Omegail/\ndcero$, except if $\ndcero = \nnum-1$, in which case $\Omegail<\nnum$ and thus $\pinf<1$, $\psup>1$; or if $\ndcero = \nnum$ and $\Omegail=\nnum$, which gives $\pinf<1$, $\psup=1$; in either case $I_\ndcero = (\pinf,1)$. According to \eqref{eq: il risk  1}, and taking into account \eqref{eq: cond Omega}, to establish \eqref{eq: teo: loss upper ratio gen} it suffices to show that, for $p \in (0,1)$ and $\ndcero$ given by \eqref{eq: ndcero Omega il},
\begin{align}
\label{eq: Y1 2}
Y_\nnum(\ndcero,p) & < \Omegail^{\nnum-1} \exp(-\Omegail) / (\nnum-1)!, \\
\label{eq: Y2 2}
Y_{\nnum+1}(\ndcero+1,p) & < \Omegail^\nnum \exp(-\Omegail) / \nnum!.
\end{align}
If $\ndcero = \nnum-1$ the left-hand sides of \eqref{eq: Y1 2} and \eqref{eq: Y2 2} are zero, and the inequalities are clearly satisfied. Thus in the following it will be assumed that $\ndcero \geq \nnum$.

As a step in the proof of \eqref{eq: Y1 2} and \eqref{eq: Y2 2}, it will be shown that for $p \in (0,1)$ and $\ndcero \geq \nnum$ related by \eqref{eq: ndcero Omega il}, or equivalently for $\ndcero \geq \nnum$ and $p \in I_\ndcero$, the following inequalities hold:
\begin{align}
\label{eq: Y 1 max}
Y_\nnum(\ndcero,p) & \leq \begin{cases}
Y_\nnum(\ndcero,(\nnum-1)/\ndcero) & \text{if } (\Omegail-\nnum+1) \ndcero \leq \nnum-1, \\
Y_\nnum(\ndcero,\Omegail/(\ndcero+1)) & \text{if } (\Omegail-\nnum+1) \ndcero > \nnum-1.
\end{cases}\\
\label{eq: Y 2 max}
Y_{\nnum+1}(\ndcero+1,p) & \leq \begin{cases}
Y_{\nnum+1}(\ndcero+1,\nnum/(\ndcero+1)) & \text{if } (\nnum-\Omegail) \ndcero \leq \Omegail, \\
Y_{\nnum+1}(\ndcero+1,\Omegail/\ndcero) & \text{if } (\nnum-\Omegail) \ndcero > \Omegail.
\end{cases}
\end{align}
For $\ndcero \geq \nnum$, it follows from \eqref{eq: Y nnum gen} that $Y_\nnum(\ndcero,p)$ considered as a function of $p \in (0,1)$ is maximum at $\pmax = (\nnum-1)/\ndcero < 1$, monotone increasing for $p < \pmax$, and monotone decreasing for $p > \pmax$. As $\Omegail > \nnum-1$, it is seen that $\pmax < \psup \leq 1$, and that $\pmax < \pinf$ if and only if $(\Omegail-\nnum+1) \ndcero > \nnum-1$. This implies that $Y_\nnum(\ndcero,p)$ is bounded as given by \eqref{eq: Y 1 max}. Regarding \eqref{eq: Y 2 max}, the maximum of $Y_{\nnum+1}(\ndcero+1,p)$ with respect to $p \in (0,1)$ is attained at $\pmax'=\nnum/(\ndcero+1)<1$. As $\ndcero \geq \nnum$, it stems that $\pinf \leq \pmax' < 1$, and that $\psup < \pmax'$ if and only if $(\nnum-\Omegail) \ndcero > \Omegail$. This establishes \eqref{eq: Y 2 max}.

The proof of \eqref{eq: Y1 2} will be based on \eqref{eq: Y 1 max}. Since $\Omegail > \nnum-1$, the following definition can be made: $\ndcerocuno = (\nnum-1)/(\Omegail-\nnum+1)$. The fact that $\Omegail \leq \nnum$ implies that $\ndcerocuno \geq \nnum-1$. The upper condition in \eqref{eq: Y 1 max} is equivalent to $\ndcero \leq \ndcerocuno$, whereas the lower corresponds to $\ndcero > \ndcerocuno$. As $\ndcero$ cannot be smaller than $\nnum$, the condition $\ndcero \leq \ndcerocuno$ can only be met for some $\ndcero$ if $\ndcerocuno \geq \nnum$, i.e.~if $\Omegail \leq \nnum-1/\nnum$. On the other hand, the condition $\ndcero > \ndcerocuno$ can always be satisfied by taking $\ndcero$ sufficiently large. Thus, \eqref{eq: Y1 2} will be established in two steps. First, it will be shown that $Y_\nnum(\ndcero,\Omegail/(\ndcero+1))$ monotonically increases with $\ndcero > \ndcerocuno$ and tends to $\Omegail^{\nnum-1} \exp(-\Omegail) / (\nnum-1)!$ as $\ndcero \rightarrow \infty$. This will prove that \eqref{eq: Y1 2} holds for all $\ndcero > \ndcerocuno$. Second, it will be shown, for $\ndcerocuno \geq \nnum$, that $Y_\nnum(\ndcero,(\nnum-1)/\ndcero)$ monotonically increases with $\ndcero \geq \nnum$ and is smaller than $\Omegail^{\nnum-1} \exp(-\Omegail) / (\nnum-1)!$ for $\ndcero = \ndcerocuno$. This will establish \eqref{eq: Y1 2} for all $m$ such that $\nnum \leq \ndcero \leq \ndcerocuno$.

Regarding the first case, $\ndcero > \ndcerocuno$, consider Lemma~\ref{lem: y nnum Taylor}\ref{lem: y nnum Taylor a} with values $\nnum$, $\ndcero$, $\Omegail$, $0$ respectively for $\rho$, $\mu$, $\Omega$, $\delta$. These values satisfy the hypotheses of the Lemma (it is obvious that (i)--(iii) hold; (iv) and (v) are satisfied as well because $\ndcero>\ndcerocuno \geq \nnum-1 \geq \Omegail-1$). According to this, $Y_\nnum(\ndcero,\Omegail/(\ndcero+1))$ monotonically increases with $\ndcero$ and tends to $\Omegail^{\nnum-1} \exp(-\Omegail) / (\nnum-1)!$ as $\ndcero \rightarrow \infty$. Therefore \eqref{eq: Y1 2} holds for $\ndcero > \ndcerocuno$.

For the case $\nnum \leq \ndcero \leq \ndcerocuno$, $\ndcerocuno \geq \nnum$, using Lemma~\ref{lem: y nnum Taylor}\ref{lem: y nnum Taylor a} (with values $\nnum$, $\ndcero$, $\nnum-1$, $-1$ respectively for $\rho$, $\mu$, $\Omega$, $\delta$; (iv) and (v) hold because $\ndcero \geq \nnum$) it is seen that $Y_\nnum(\ndcero,(\nnum-1)/\ndcero)$ increases with $\ndcero$. The definition of $\ndcerocuno$ implies that $(\nnum-1)/\ndcerocuno = \Omegail/(\ndcerocuno+1)$, and thus
\begin{equation}
\label{eq: Y1 ndcero menor}
Y_\nnum( \ndcerocuno, (\nnum-1)/\ndcerocuno ) = Y_\nnum( \ndcerocuno, \Omegail/(\ndcerocuno+1) ).
\end{equation}
Applying Lemma~\ref{lem: y nnum Taylor}\ref{lem: y nnum Taylor a} again (with values $\nnum$, $\ndcerocuno$, $\Omegail$, $0$; note that (v) is satisfied because $\ndcerocuno \geq \nnum \geq \Omegail > \Omegail-1$) to the right-hand side of this equality shows that \eqref{eq: Y1 ndcero menor} is smaller than $\Omegail^{\nnum-1} \exp(-\Omegail) / (\nnum-1)!$. Therefore \eqref{eq: Y1 2} holds for $\nnum \leq \ndcero \leq \ndcerocuno$.

As for \eqref{eq: Y2 2}, it is seen that the lower condition in \eqref{eq: Y 2 max} is not met for any $\ndcero$ if $\Omegail=\nnum$, whereas if $\Omegail<\nnum$ there exist values of $\ndcero$ which satisfy each of the conditions. These two cases will be treated separately.

In the case $\Omegail<\nnum$, the proof proceeds along the same lines as that of \eqref{eq: Y1 2}. Let $\ndcerocdos = \Omegail/(\nnum-\Omegail)$. The fact that $\Omegail > \nnum-1$ implies that $\ndcerocdos > \nnum-1$. In addition, since $\ndcero \geq \nnum$, the upper condition in \eqref{eq: Y 2 max} can only be met if $\ndcerocdos \geq \nnum$. Thus it suffices to show first that $Y_{\nnum+1}(\ndcero+1,\Omegail/\ndcero)$ monotonically increases with $\ndcero > \ndcerocdos$ and tends to $\Omegail^\nnum \exp(-\Omegail) / \nnum!$ as $\ndcero \rightarrow \infty$; and second that, if $\ndcerocdos \geq \nnum$, $Y_{\nnum+1}(\ndcero+1,\nnum/(\ndcero+1))$ monotonically increases with $\ndcero \geq \nnum$ and is smaller than $\Omegail^\nnum \exp(-\Omegail) / \nnum!$ for $\ndcero = \ndcerocdos$. The first part directly stems from Lemma~\ref{lem: y nnum Taylor}\ref{lem: y nnum Taylor b} (with values $\nnum$, $\ndcero$, $\Omegail$, $-1$). As for the second, the increasing character of $Y_{\nnum+1}(\ndcero+1,\nnum/(\ndcero+1))$ with $\ndcero$ is also established by Lemma~\ref{lem: y nnum Taylor}\ref{lem: y nnum Taylor b} (with values $\nnum$, $\ndcero$, $\nnum$, $0$). The definition of $\ndcerocdos$ implies that $\nnum/(\ndcerocdos+1) = \Omegail/\ndcerocdos$, from which
\begin{equation}
\label{eq: Y2 ndcero menor}
Y_{\nnum+1}( \ndcerocdos+1, \nnum/(\ndcerocdos+1) ) = Y_{\nnum+1} ( \ndcerocdos+1, \Omegail/\ndcerocdos ),
\end{equation}
and applying Lemma~\ref{lem: y nnum Taylor}\ref{lem: y nnum Taylor b} (with values $\nnum$, $\ndcerocdos$, $\Omegail$, $-1$; (iv) and (v) hold because $\ndcerocdos = \Omegail/(\nnum-\Omegail) > \Omegail > \nnum-1$) to the right-hand side of \eqref{eq: Y2 ndcero menor} establishes that it is smaller than $\Omegail^\nnum \exp(-\Omegail) / \nnum!$.

In the case $\Omegail=\nnum$, the expression \eqref{eq: Y 2 max} reduces to its upper part, and \eqref{eq: Y2 2} follows from Lemma~\ref{lem: y nnum Taylor}\ref{lem: y nnum Taylor b} (with values $\nnum$, $\ndcero$, $\nnum$, $0$). This completes the proof.
\end{proof}

\begin{proof}[Proof of Proposition~\ref{prop: a b para minimax}]
For $L$ as in \eqref{eq: loss abs error gen}, equating $\diff(\lim_{p \rightarrow 0} \risk(p))/\diff\Omega$ given by \eqref{eq: diff bar risk diff Omega ll} to $0$, solving for $a/b$ and particularizing to $\Omega=\nnum-1$ yields \eqref{eq: a b minimax ll}.

As for $L$ given by \eqref{eq: loss upper ratio gen}, from \eqref{eq: cond Omega} it is seen that
\begin{equation}
\label{eq: a b minimax il parte2}
\frac a b = \frac{(\Omegail-\nnum+1)\Omegail}{(\nnum-1)(\nnum-\Omegail)}.
\end{equation}
Setting $\Omega^*=\Omegail$ in \eqref{eq: Omega opt il} and combining with \eqref{eq: a b minimax il parte2} yields \eqref{eq: a b minimax il}.
\end{proof}

\begin{lemma}
\label{lemma: Stirling}
For any $k \in \mathbb N$, the factorial $k!$ satisfies the following:
\begin{align}
k! & > \sqrt{2\pi} k^{k+1/2} \exp(-k), \\
\lim_{k \rightarrow \infty} \frac{k! \exp(k)}{k^{k+1/2}} & = \sqrt{2\pi}.
\end{align}
\end{lemma}

\begin{proof}
These expressions follow from \citet[eq.~(6.1.38)]{Abramowitz70}.
\end{proof}

\begin{lemma}
\label{lemma: gamma iguales lim}
For any sequence of numbers $\delta_k$ such that $0 \leq \delta_k \leq 1$, $\lim_{k \rightarrow \infty} \gamma(k,k+\delta_k) = 1/2$.
\end{lemma}

\begin{proof}
According to \citet[lemma~1]{Adell05},
$\lim_{k \rightarrow \infty} \gamma(k,k) = 1/2$. From \eqref{eq: gamma vecinas},
\begin{equation}
\label{eq: gamma k,k+1}
\gamma(k,k+1) - \gamma(k+1,k+1) = {(k+1)^k \exp(-k-1)} / {k!}.
\end{equation}
As a result of Lemma~\ref{lemma: Stirling}, the right-hand side of \eqref{eq: gamma k,k+1} tends to $0$ as $k \rightarrow \infty$, and therefore $\lim_{k \rightarrow \infty} \gamma(k,k+1) = \lim_{k \rightarrow \infty} \gamma(k,k) = 1/2$. The fact that $\gamma(t,u)$ is monotone increasing in $u$ implies that $\gamma(k,k) \leq \gamma(k,k+\delta_k) \leq \gamma(k,k+1)$, and the desired result follows.
\end{proof}

\begin{lemma}
\label{lem: Omega opt linlin cota}
For $\nnum \geq 2$ and $a=b$, the solution $\Omega^*$ to \eqref{eq: Omega opt ll} lies in $(\nnum-4/3,\nnum-2+\log 2)$, and $\lim_{\nnum \rightarrow \infty} (\Omega^*-\nnum) = -4/3$.
\end{lemma}

\begin{proof}
The result follows from \citet{Alm03}.
\end{proof}

\begin{lemma}
\label{lem: Omega opt invlin cota}
For $\nnum \geq 2$ and $a=b$, the solution $\Omega^*$ to \eqref{eq: Omega opt il} lies in $(\nnum-1, \nnum)$.
\end{lemma}

\begin{proof}
Using \eqref{eq: gamma vecinas} the condition \eqref{eq: Omega opt il} can be written, for $a=b$, as
\begin{equation}
\label{eq: Omega opt il 2}
\left( \frac{{\Omega^*}^2}{\nnum(\nnum-1)} + 1 \right) \gamma(\nnum,\Omega^*) = \frac{{\Omega^*}^\nnum \exp(-\Omega^*)}{\nnum!} \left( 1 - \frac{\Omega^*}{\nnum-1} \right) + 1.
\end{equation}
Let $v_1(\Omega^*)$ and $v_2(\Omega^*)$ respectively denote the left-hand and right-hand sides of \eqref{eq: Omega opt il 2}, considered as functions of $\Omega^*$. It is easily seen that $v_1$ is monotone increasing, whereas $v_2$ is monotone decreasing on the interval $(\nnum-1,\nnum)$. From Lemma~\ref{lem: Omega opt linlin cota} and the monotonicity of $\gamma(t,u)$ with respect to $u$ it follows that $\gamma(\nnum,\nnum-1) < 1/2$, which implies that $v_1(\nnum-1) < 1$. On the other hand, $v_2(\nnum-1) = 1$. Therefore the solution to \eqref{eq: Omega opt il 2}, or equivalently to \eqref{eq: Omega opt il}, satisfies $\Omega^* > \nnum-1$. By analogous arguments it is seen that $v_1(\nnum) > 1$ and $v_2(\nnum) < 1$. Therefore the solution satisfies $\Omega^* < \nnum$.
\end{proof}

\begin{lemma}
\label{lemma: conv unif}
For any $\delta_1, \delta_2 \in \mathbb R$, the sequence of functions $h_k(\delta) = \exp(\delta) ( 1 + \delta / (k-1) )^{-k+1}$, $k \in \mathbb N$, $ k \geq 2$, $\delta \in [\delta_1,\delta_2]$ converges uniformly to $1$ as $k \rightarrow \infty$.
\end{lemma}

\begin{proof}
Let $k_0 = \max\{-\delta_1,0\}+2$. As $1+\delta/(k-1) > 0$ for $\delta \geq \delta_1$, $k \geq k_0$, it is possible to take logarithms in the definition of $h_k(\delta)$ for $k \geq k_0$, which gives $\log h_k(\delta) = \delta - (k-1) \log( 1 + \delta / (k-1) )$. Replacing $k$ by a continuous variable $x>1$ and using the inequality $\log(1+t) > t/(1+t)$, it is seen that
\begin{equation}
\label{eq: der log h k}
\frac{\partial}{\partial x} \left[ \delta - (x-1) \log \left( 1 + \frac{\delta}{x-1} \right) \right] = -\log \left( 1 + \frac{\delta}{x-1} \right) + \frac{\delta}{x+\delta-1} < 0.
\end{equation}
This implies that $h_{k+1}(\delta) < h_k(\delta)$ for $k \geq k_0$. In addition, $h_k(\delta)$, $\delta \in [\delta_1, \delta_2]$ is a continuous function and converges pointwise to $1$ as $k \rightarrow \infty$. Thus Dini's theorem \citep[p.~248]{Apostol74} can be applied, which ensures that the convergence is uniform.
\end{proof}

\begin{proof}[Proof of Proposition~\ref{prop: a b 1 asint minimax}]
For $L$ as in \eqref{eq: loss abs error gen}, particularizing \eqref{eq: bar risk ll} to $a=b$, $\Omega=\nnum-1$ and using \eqref{eq: gamma vecinas},
\begin{equation}
\label{eq: bar risk ll sym}
\frac{\bar\risk}{a} = 2(\gamma(\nnum-1,\nnum-1) - \gamma(\nnum,\nnum-1)) = \frac{2(\nnum-1)^{\nnum-2} \exp(-\nnum+1)}{(\nnum-2)!}.
\end{equation}
In the following, the value $\Omega^*$ determined by \eqref{eq: Omega opt ll} for a given $\nnum$ will be denoted as $\Omega^*_\nnum$. Particularizing \eqref{eq: risk opt ll} to $a=b$,
\begin{equation}
\label{eq: bar risk ll sym opt}
\frac{\risk^*}{a}
= \frac{2 {\Omega^*_\nnum}^{\nnum-1} \exp(-\Omega^*_\nnum) }{(\nnum-1)!}.
\end{equation}
From \eqref{eq: bar risk ll sym} and \eqref{eq: bar risk ll sym opt}, with $h_k(\delta)$ as defined in Lemma~\ref{lemma: conv unif}, it follows that
\begin{equation}
\label{eq: rel risk h}
\frac{\bar\risk}{\risk^*} = \left( \frac{\nnum-1}{\Omega^*_\nnum} \right)^{\nnum-1} \exp(\Omega^*_\nnum-\nnum+1) = h_\nnum(\delta^*_\nnum)
\end{equation}
with $\delta^*_\nnum = \Omega^*_\nnum-\nnum+1$. Lemma~\ref{lem: Omega opt linlin cota} establishes that $\delta^*_\nnum \in [-1/3, -1+\log 2]$ and $\lim_{\nnum \rightarrow \infty} \delta^*_\nnum =  -1/3$. On the other hand, by Lemma~\ref{lemma: conv unif}, $h_k \rightarrow 1$  uniformly on $[-1/3, -1+\log 2]$ as $k \rightarrow \infty$. Therefore, according to \citet[theorem 9.16]{Apostol74},
$\lim_{k,l \rightarrow \infty} h_k(\delta^*_l)$ exists and equals $1$. Thus, in particular, $\lim_{\nnum \rightarrow \infty} h_\nnum(\delta^*_\nnum) = 1$, which combined with \eqref{eq: rel risk h} establishes that $\lim_{\nnum \rightarrow \infty} \bar\risk/\risk^* = 1$.

For $L$ as in \eqref{eq: loss upper ratio gen}, and with $\Omega^*$ given by \eqref{eq: Omega opt il}, let $\Omega^*_\nnum$ and $\delta^*_\nnum$ be defined as before. In addition, let $\Omegail_\nnum$ denote the value of $\Omegail$ corresponding to a given $\nnum$, and $\deltail_\nnum = \Omegail_\nnum-\nnum+1$. Particularizing \eqref{eq: bar risk il} to $a=b$, $\Omega=\Omegail_\nnum$ and using \eqref{eq: gamma vecinas} gives
\begin{equation}
\label{eq: bar risk il sym}
\begin{split}
\frac{\bar\risk}{a} & = \frac{\nnum}{\Omegail_\nnum} ( \gamma (\nnum-1,\Omegail_\nnum) - \gamma (\nnum+1,\Omegail_\nnum) ) + \left( \frac{\Omegail_\nnum}{\nnum-1} - \frac{\nnum}{\Omegail_\nnum} \right) \gamma(\nnum-1,\Omegail_\nnum) + \frac{\nnum}{\Omegail_\nnum}-1 \\
& = \frac{\Omegail_\nnum^{\nnum-1} \exp(-\Omegail_\nnum)}{(\nnum-1)!} \left( 1 + \frac{\nnum}{\Omegail_\nnum} \right)
+ \left( \frac{\Omegail_\nnum}{\nnum-1} - \frac{\nnum}{\Omegail_\nnum} \right) \gamma(\nnum-1,\Omegail_\nnum) + \frac{\nnum}{\Omegail_\nnum}-1.
\end{split}
\end{equation}
Thus $\bar\risk$ can be written as $a(\riskparteil_0 + \riskparteil_1 + \riskparteil_2)$ with
\begin{align}
\label{eq: riskparte0}
\riskparteil_0 & = \frac{(\nnum-1+\deltail_\nnum)^{\nnum-2} \exp(-\nnum+1-\deltail_\nnum)}{(\nnum-2)!} \left( 2 + \frac{1+\deltail_\nnum}{\nnum-1} \right), \\
\riskparteil_1 &
= \frac{ \frac{\deltail_\nnum^2}{\nnum-1} + 2\deltail_\nnum - 1 }{\nnum-1+\deltail_\nnum} \gamma(\nnum-1,\nnum-1+\deltail_\nnum), \\
\label{eq: riskparte2}
\riskparteil_2 & = \frac{1-\deltail_\nnum}{\nnum-1+\deltail_\nnum}.
\end{align}
The quotient $\riskparteil_2/\riskparteil_0$ is computed as
\begin{equation}
\label{eq: riskparte2 div 0}
\frac{\riskparteil_2}{\riskparteil_0}
= \frac{(1-\deltail_\nnum) \exp(\deltail_\nnum)}{ \left( 1 + \frac{\deltail_\nnum}{\nnum-1} \right)^{\nnum-1} \left( 2 + \frac{1+\deltail_\nnum}{\nnum-1} \right) } \cdot
\frac{(\nnum-1)! \exp(\nnum-1)}{(\nnum-1)^{\nnum-1/2}} \cdot
\frac{1}{\sqrt{\nnum-1}}.
\end{equation}
Proposition~\ref{prop: Omega invlin cota} implies that $\deltail_\nnum \in (0,1)$. Taking into account that $\nnum \geq 2$, it is seen that the first factor in \eqref{eq: riskparte2 div 0} lies in a bounded interval for all $\nnum$, whereas, by the equality in Lemma~\ref{lemma: Stirling}, the second factor tends to $\sqrt{2\pi}$ as $\nnum \rightarrow \infty$. As a result, $\lim_{\nnum \rightarrow \infty} \riskparteil_2/\riskparteil_0 = 0$. Similarly, $\riskparteil_1/\riskparteil_0$ is expressed as
\begin{equation}
\label{eq: riskparte1 div 0}
\frac{\riskparteil_1}{\riskparteil_0}
= \frac{ \left(  \frac{\deltail_\nnum^2}{\nnum-1} + 2\deltail_\nnum - 1 \right) \exp(\deltail_\nnum) }
{ \left( 1+\frac{\deltail_\nnum}{\nnum-1}\right)^{\nnum-1} \left( 2+\frac{1+\deltail_\nnum}{\nnum-1}\right) } \cdot
\frac{(\nnum-1)! \exp(\nnum-1)}{(\nnum-1)^{\nnum-1/2}} \cdot
\gamma(\nnum-1,\nnum-1+\deltail_\nnum) \cdot
\frac{1}{\sqrt{\nnum-1}}.
\end{equation}
As before, the first factor in the right-hand side of \eqref{eq: riskparte1 div 0} is bounded, and the second tends to $\sqrt{2\pi}$. The third factor tends to $1/2$ by Lemma~\ref{lemma: gamma iguales lim}. Thus $\lim_{\nnum \rightarrow \infty} \riskparteil_1/\riskparteil_0 = 0$.

The quotient $\risk^*/a$ is given as in \eqref{eq: bar risk il sym} with $\Omegail_\nnum$ replaced by $\Omega^*_\nnum$; and $\risk^* = a(\riskparte_0^*+\riskparte_1^*+\riskparte_2^*)$, where $\riskparte_0^*$, $\riskparte_1^*$ and $\riskparte_2^*$ are obtained from \eqref{eq: riskparte0}--\eqref{eq: riskparte2}
with $\deltail_\nnum$ replaced by $\delta^*_\nnum$. Lemma~\ref{lem: Omega opt invlin cota} implies that $\delta^*_\nnum \in (0,1)$, and arguments analogous to those in the preceding paragraph show that $\riskparte_1^*/\riskparte_0^*$ and $\riskparte_2^*/\riskparte_0^*$ tend to $0$ as $\nnum \rightarrow \infty$. As a result, $\lim_{\nnum \rightarrow \infty} \bar\risk/\risk^*$ can be computed as
\begin{equation}
\label{eq: riskparte0 div}
\begin{split}
\lim_{\nnum \rightarrow \infty} \frac{\bar\risk}{\risk^*}
& = \lim_{\nnum \rightarrow \infty} \frac{\riskparteil_0}{\riskparte_0^*}
= \lim_{\nnum \rightarrow \infty}
\frac{\left(1+\frac{\deltail_\nnum}{\nnum-1} \right)^{\nnum-2} \exp(-\deltail_\nnum)}{\left(1+\frac{\delta^*_\nnum}{\nnum-1}\right)^{\nnum-2} \exp(-\delta^*_\nnum)}
\cdot
\frac{2+\frac{1+\deltail_\nnum}{\nnum-1}}{2+\frac{1+\delta^*_\nnum}{\nnum-1}}.
\end{split}
\end{equation}
Since $\deltail_\nnum, \delta^*_\nnum \in (0,1)$ for all $\nnum$, it is clear that the second factor in the rightmost part of \eqref{eq: riskparte0 div} tends to $1$ as $\nnum \rightarrow \infty$. By Lemma~\ref{lemma: conv unif}, $(1+\delta/(\nnum-1))^{\nnum-1} \exp(-\delta) \rightarrow 1$ uniformly for $\delta \in (0,1)$. This implies that the numerator and denominator of the first factor in \eqref{eq: riskparte0 div} tend to $1$ as $\nnum \rightarrow \infty$ (note that $\deltail_\nnum$ and $\delta^*_\nnum$ are not required to converge). Consequently $\lim_{\nnum \rightarrow \infty} \bar\risk/\risk^* = 1$.
\end{proof}


\begin{thebibliography}{25}
\expandafter\ifx\csname natexlab\endcsname\relax\def\natexlab#1{#1}\fi

\bibitem[{Abramowitz and Stegun(1970)}]{Abramowitz70}
Abramowitz M, Stegun IA (eds)  (1970) Handbook of Mathematical Functions, ninth
  edn. Dover

\bibitem[{Adell and Jodr\'a(2005)}]{Adell05}
Adell JA, Jodr\'a P (2005) The median of the {Poisson} distribution. Metrika
  61:337--346

\bibitem[{Akdeniz(2004)}]{Akdeniz04}
Akdeniz F (2004) New biased estimators under the linex loss function.
  Statistical Papers 45:175--190

\bibitem[{Alm(2003)}]{Alm03}
Alm SE (2003) Monotonicity of the difference between median and mean of gamma
  distributions and of a related {Ramanujan} sequence. Bernoulli 9(2):351--371

\bibitem[{Alvo(1977)}]{Alvo77}
Alvo M (1977) Bayesian sequential estimation. Annals of Statistics
  5(5):955--968

\bibitem[{Apostol(1974)}]{Apostol74}
Apostol TM (1974) Mathematical Analysis, 2nd edn. Addison-Wesley

\bibitem[{Baran and Magiera(2010)}]{Baran10}
Baran J, Magiera R (2010) Optimal sequential estimation procedures of a
  function of a probability of success under {LINEX} loss. Statistical Papers
  51(3):511--529

\bibitem[{Berger(1985)}]{Berger85}
Berger JO (1985) Statistical Decision Theory and {Bayesian} Analysis, 2nd edn.
  Springer-Verlag

\bibitem[{Best(1974)}]{Best74}
Best DJ (1974) The variance of the inverse binomial estimator. Biometrika
  61(2):385--386

\bibitem[{Cabilio(1977)}]{Cabilio77}
Cabilio P (1977) Sequential estimation in {Bernoulli} trials. Annals of
  Statistics 5(2):342--356

\bibitem[{Cabilio and Robbins(1975)}]{Cabilio75}
Cabilio P, Robbins H (1975) Sequential estimation of $p$ with squared relative
  error loss. Proceedings of the National Academy of Sciences of the United
  States of America 72(1):191--193

\bibitem[{Christoffersen and Diebold(1997)}]{Christoffersen97}
Christoffersen PF, Diebold FX (1997) Optimal prediction under asymmetric loss.
  Econometric Theory 13:808--817

\bibitem[{DeGroot(1959)}]{DeGroot59}
DeGroot MH (1959) Unbiased sequential estimation for binomial populations.
  Annals of Mathematical Statistics 30(1):80--101

\bibitem[{Girshick et~al(1946)Girshick, Mosteller, and Savage}]{Girshick46}
Girshick MA, Mosteller F, Savage LJ (1946) Unbiased estimates for certain
  binomial sampling problems with applications. Annals of Mathematical
  Statistics 17(1):13--23

\bibitem[{Granger(1969)}]{Granger69}
Granger CWJ (1969) Prediction with a generalized cost of error function.
  Operational Research Quarterly 20(2):199--207

\bibitem[{Haldane(1945)}]{Haldane45}
Haldane JBS (1945) On a method of estimating frequencies. Biometrika
  33(3):222--225

\bibitem[{Hubert and Pyke(2000)}]{Hubert00}
Hubert SL, Pyke R (2000) Sequential estimation of functions of $p$ for
  {Bernoulli} trials. In: Game Theory, Optimal Stopping, Probability and
  Statistics, Institute of Mathematical Statistics, pp 263--294

\bibitem[{Lehmann and Casella(1998)}]{Lehmann98}
Lehmann EL, Casella G (1998) Theory of Point Estimation, 2nd edn. Springer

\bibitem[{Mendo(2009)}]{Mendo09a}
Mendo L (2009) Estimation of a probability with guaranteed normalized mean
  absolute error. {IEEE} Communications Letters 13(11):817--819

\bibitem[{Mendo(2012)}]{Mendo10_env}
Mendo L (2012) Asymptotically optimum estimation of a probability in inverse
  binomial sampling. Journal of Statistical Planning and Inference 142(10):2862--2870.

\bibitem[{Mendo and Hernando(2006)}]{Mendo06}
Mendo L, Hernando JM (2006) A simple sequential stopping rule for {Monte Carlo}
  simulation. {IEEE} Transactions on Communications 54(2):231--241

\bibitem[{Mendo and Hernando(2008)}]{Mendo08a}
Mendo L, Hernando JM (2008) Improved sequential stopping rule for {Monte Carlo}
  simulation. {IEEE} Transactions on Communications 56(11):1761--1764

\bibitem[{Mendo and Hernando(2010)}]{Mendo10}
Mendo L, Hernando JM (2010) Estimation of a probability with optimum guaranteed
  confidence in inverse binomial sampling. Bernoulli 16(2):493--513

\bibitem[{Mikulski and Smith(1976)}]{Mikulski76}
Mikulski PW, Smith PJ (1976) A variance bound for unbiased estimation in
  inverse sampling. Biometrika 63(1):216--217

\bibitem[{Sathe(1977)}]{Sathe77}
Sathe YS (1977) Sharper variance bounds for unbiased estimation in inverse
  sampling. Biometrika 64(2):425--426

\end{thebibliography}
\end{document}